\documentclass[10pt]{article}
\usepackage{amsmath}
\usepackage{amssymb}
\usepackage{amsthm}
\usepackage{amscd}
\usepackage{graphicx}
\usepackage[table]{xcolor}

\usepackage{lipsum} 
\usepackage{fancyhdr}
\newtheorem{thm}{Theorem}
\newtheorem{lem}{Lemma}
\newtheorem{prop}{Proposition}
\newtheorem{cor}{Corollary}

\newtheorem{conj}{Conjecture}

\newcommand{\ignore}[1]{}
\newcommand{\ol}{\overline}
\newcommand{\Phip}{\mathbf{\Phi}^{1/2}}
\newcommand{\Phim}{\mathbf{\Phi}^{-1/2}}

\newcommand{\bA}{\mathbf{A}}
\newcommand{\bD}{\mathbf{D}}

\newcommand{\bI}{\mathbf{I}}
\newcommand{\bM}{\mathbf{M}}
\newcommand{\bP}{\mathbf{P}}
\newcommand{\bone}{\mathbf{1}}

\newcommand{\bu}{\mathbf{u}}
\newcommand{\bv}{\mathbf{v}}

\newcommand{\bU}{\mathbf{U}}
\newcommand{\bx}{\mathbf{x}}

\newcommand{\bJ}{\mathbf{J}}

\newcommand{\by}{\mathbf{y}}

\begin{document}
\title{Concentration of the Stationary Distribution on General Random Directed Graphs}
\author{Franklin H.J. Kenter\thanks{fhk2@rice.edu \newline
 Department of Computational and Applied Mathematics; Rice University; Houston, TX 77005}}
\maketitle

%


%
%

\begin{abstract}
We consider a random model for directed graphs whereby an arc is placed from one vertex to another with a prescribed probability which may vary from arc to arc. Using perturbation bounds as well as Chernoff inequalities, we show that the stationary distribution of a Markov process on a random graph is concentrated near that of the ``expected'' process under mild conditions. These conditions involve the ratio between the minimum and maximum in- and out-degrees, the ratio of the minimum and maximum entry in the stationary distribution, and the smallest singular value of the transition matrix. Lastly, we give examples of applications of our results to well-known models such as PageRank and $G(n,p).$
%
%
\end{abstract}

\section{Introduction}
The study of random matrices, their spectra, and eigenvectors has a long and rich history.
Much focus has been placed on the case where the entries are identically and independently  distributed- for which the distribution of eigenvalues, and even eigenvectors, is well-known \cite{bai, KnowlesJin, tv1, tv2}.

The inspiration for this work comes from beyond of random matrix theory. Over the past decade, there have been various Markov chain ranking algorithms. Specifically, these algorithms use data in order to construct a Markov chain, then use the components of the stationary distribution in order to rank and compare elements. Examples of such algorithms include PageRank  by Brin and Page \cite{LargeScale}, and HITS  by Kleinburg \cite{HITS} for web search;  NCAA Random Walkers by Callaghan, Mucha, and Porter \cite{NCAARW}, and Logarithmic Regression Markov Chain (LRMC) by Kvam and Sokol \cite{LRMC} for sports ranking; and NetRank for cancer research \cite{Cancer}. Several of these algorithms have demonstrated to be quite effective.  LRMC is able to predict basketball game outcomes more accurately than statistical methods \cite{LRMC}. Additionally, NetRank was capable of finding previously unknown gene markers related to pancreatic cancer \cite{Cancer}. Despite these successes, there does not seem to be any statistical nor probabilistic indication as to \emph{why} these algorithms work so well. Using a random graph model, we aim to develop to establish a justification as to why these types of algorithms can be effective.

Our random graph model, described in detail in the next section, places arcs from vertex $i$ to vertex $j$ with a predetermined probability, for each pair independently. These probabilities may differ from pair to pair, and may differ in direction. The undirected version of this model, with slight variations has been studied by Chung, Lu, and Vu \cite{FLV1,FLV2}, Oliveira \cite{O}, Chung and Radcliffe \cite{Mary}, and Lu and Peng \cite{LuPeng}. 

In the works mentioned above, the study focused on the eigenvalues of the adjacency matrix for random graphs and not the eigenvectors, as the principal eigenvector for undirected graphs corresponds to the degrees of the vertices. Hence, for undirected graphs, the entries for the principal eigenvector of the adjacency matrix can only be quadratically small and not interesting to study. However, in the directed case, there are examples of almost regular directed graphs whose principle eigenvector (i.e., stationary distribution) has exponentially small entries (we give such an example in section \ref{examplesGR}). In particular, the principle eigenvector is extremely sensitive to arc-deletion: the removal of just a few arcs can change the principle eigenvector from a uniform distribution to having exponentially small entries.

In this paper, we give criteria as to when the stationary distribution of a Markov chain on a general random directed graph  is close to the ``expected'' stationary distribution. We show that such concentration occurs when the range of in- and out-degrees as well as values of the ``expected'' stationary distribution are well-behaved and when the spectral gap is small. A full statement can be found in Theorem \ref{mainresult}. 

This paper is organized as follows: In Section \ref{pre} we give the preliminaries and notation. Next, inSection \ref{main}, we describe the main result.  In following two sections, we develop tools needed to prove our main results: perturbation bounds in Section \ref{Cauch} and probabilistic concentration inequalities in Section \ref{concen}. Finally, we prove our result using these tools in Section \ref{mainproof} and provide examples in Sections \ref{examplesGR}, \ref{exampleGNP}, and \ref{examplePR}.

\section{Preliminaries}\label{pre}

Let $\bI$ and $\mathbf{0}$ denote the identity and zero matrices, respectively. Let $\bv^*$ and $\bM^*$ denote the conjugate transpose of a vector $\bv$ and matrix $\bM$, respectively. Throughout, all vectors are column vectors. A right eigenvector of $\bM$ with eigenvalue $\lambda$ is a vector $\bv$ such that $\bM \bv = \lambda \bv$ for some $\lambda \in \mathbb{C}$. For our purposes, a left eigenvector with eigenvalue $\lambda$ is a vector $\bv$ such that $\bv^* \bM = \lambda \bv^*$. A left generalized eigenvector with eigenvalue $\lambda$ is a vector $\bv$ such that $ \bv^* (\bM - \lambda \bI)^k = \mathbf{0}$ for some positive integer $k$.

For a directed graph $G = (V, E)$, we denote an arc from vertex $i$ to $j$ as $i \to j$. We  denote the in- and out-degrees of vertex $i$ as $d_i^{in}$ and $d_i^{out}$ respectively and the (not necessarily symmetric) adjacency matrix as $\bA$.
We consider the associated random walk on $G$ whose $n \times n$ transition  probability matrix, denoted $\bP$, given by:

$$\bP_{ij} = \frac{1}{d^{out}_i}$$

We can also express $\bP$ as $\bD^{-1} \bA$ where $\bD = diag(d_1^{out}, \ldots, d_n^{out})$ is the out-degree diagonal matrix. A stationary distribution of a Markov chain is a distribution $\phi$ such that for any non-negative vector $\bv \in \mathbb{R}^n$ that sums to 1, $\lim_{k \to \infty} \bv^* \bP^k = \phi$. Under the mild conditions of irreducibility and ergodicity, the stationary distribution exists and is unique \cite{af}. For the remainder, we shall assume that our Markov chains are irreducible and ergodic.



We will use various norms in order to compare distributions, we will let  $\| \mathbf{x} \|_1$, $\| \mathbf{x} \|_2$,  $\| \mathbf{x} \|_\infty$ denote the 1-, 2-, and $\infty$-norm of $\bx$ respectively. When the subscript is omitted, it will denote the 2-norm unless otherwise specified. Additionally, given a distribution vector $\phi$ (usually a stationary distribution), we will let $\|\bx\|_{2,\phi}$ denote the {\it chi-squared norm with respect to} $\phi$ defined as:

\[ \|\bx\|_{2,\phi}^2 :=  \sum_i \frac{\bx_i^2}{\phi_i} = \|\bx^* \Phim \|_2^2 \]

where $\Phip = {diag}~\phi$ is the diagonal matrix of $\phi$. The chi-squared norm is equivalent to several other norms such as total-variation and relative point wise distance \cite{af, SGT}.

For a random variable $X$ we let $\mathbb{E}X$ denote the expectation of $X$, and for any event $A$ we will denote the its probability $\mathbb{P}(A)$. 

Throughout, we will consider a random graph model with different edge weights. Let $\ol{\bA}$ be the {\it expected adjacency matrix} where $0 \le \ol{\bA}_{i,j} \le 1$ for all $i, j$. Given $\ol{\bA}$ we create a random adjacency matrix, $\bA$, where $\bA_{i,j} = 1$ with probability $\bA_{i,j}$ and $\bA_{i,j}=0$ otherwise, independent of all other arcs. Hence, $\mathbb{E} \bA_{i,j} = \ol{\bA}_{i,j}$.
Using the bar notation, we let $\ol d_i^{in}$ and $\ol d_i^{out}$ denote the average in- and out-degrees for vertex $i$, and $\ol d_{min}^{in}$, $\ol d_{max}^{in}$, $\ol d_{min}^{out}$, $\ol d_{max}^{out}$ to denote the minimum/maximum average in-/out-degrees, correspondingly, and $\ol d_{min} = \min [\ol d_{min}^{in}, \ol d_{min}^{out}]$ and $\ol d_{max} = \min [\ol d_{max}^{in}, \ol d_{max}^{out}]$.
Likewise, we define $\ol \bD = diag(\ol d_1^{out}, \ldots, \ol d_n^{out})$ to be the \emph{expected diagonal out-degree matrix}. We also have the \emph{``expected'' probability matrix}, $\ol \bP \colon= \ol \bD^{-1} \ol \bA$. Note that the entries of $\ol \bP$ are not necessarily the average values of $\bP$.

Given a matrix $\bA$, the {\it pseudoinverse} of $\bA$, denoted $\bA^+$, is the unique matrix such that $\bA \bA^+ \bA = \bA$ and $\bA^+ = \bA^+ \bA \bA^+$, and we will use $0 \le \sigma_1(\bM) \le \sigma_2 (\bM) \le \ldots \le \sigma_n(\bM)$ to denote the singular values of $\bM$. By the singular value decomposition, $\|\bM^+ \|$ is $\sigma_i(\bM)^{-1}$ where $\sigma_i(\bM)$ is the smallest non-zero singular value.

Throughout we use ``little-o'' and $\gg$ asymptotic notation. We say $f(x) = o(g(x))$ or, equivalently,$f(x) \gg g(x)$ if $\lim_{x\to \infty} \frac{g(x)}{f(x)} = 0.$

\section{Main Result}\label{main}

\begin{thm}[Concentration of $\phi$ to $\ol \phi$]\label{mainresult}
Let $\bA$ denote an $n \times n$ random (0,1)-matrix where each entry is a Bernoulli random variable whose parameter is given by the corresponding entry in a expected degree matrix $\ol \bA$. Let $\bP$ be the (random) Markov transition matrix corresponding to $\bA$ with stationary distribution vector $\phi$. Let $\ol \bP = \ol \bD^{-1} \ol \bA$ with left eigenvector $\ol \phi$ with eigenvalue 1 where $\ol \bD$ is the expected out-degree matrix. If the expected minimum in- and out-degrees obey $\ol d_{min} \gg n^k$ for some $k>0$, then with probability at least 

\[ 1 - 4 n  \exp\left(\frac{- d_{min}^{1/3}}{3} \right). \]

we have
\begin{eqnarray*}
&&\| \phi - \ol \phi \|_{2, \ol \phi} \le \\ 
 &&\left(1+\frac{1}{\ol d_{min}^{1/3}} \right) \left [ \frac{\left(\ol d^{in}_{max} \ol d^{out}_{max} \right)^{1/4}}{\ol d^{out}_{min}} + \frac{\sqrt{\ol d^{in}_{max} \ol d^{out}_{max}}}{{\ol d_{min}^{4/3}}}  \right ] \left(\frac{ \ol\phi_{max}}{\ol\phi_{min}}\right)^{1/2} \frac{1}{\sigma_2(\bI-\bP)}
\end{eqnarray*}

where $\sigma_2(\bI-\bP)$ is the second smallest singular value of $\bI-\bP$.

\end{thm}

We now discuss the role of the four terms. The first term, due to the hypothesis $d_{min} \gg n^k$, is asymptotically 1. The second term is the only term with the possibility of being substantially smaller than 1. Most notably, for a dense graph (i.e., the degrees are $O(n)$), then the second term is $O(n^{-1/2})$. The third term measures the spread of the components in the stationary distribution. The third term can be exponentially large  (see for example Gambler's ruin in Section \ref{examplesGR}. However, in many cases it can be quite small. For example, in an undirected graph, the third term is simply $\sqrt{\frac{ d_{max}}{ d_{min}}}$, which is 1 if the graph is regular. The fourth term is different. First, we must emphasize that the matrix $\bI - \bP$ in the statement of the result, is in fact a random matrix. Hence, the result does not assure that any randomly generated graph has a concentrated stationary distribution. Rather, the result gives a test to see if a randomly generated graph has the ``expected'' stationary distribution. In particular, one could positively determine if a Markov ranking algorithm (such as those mentioned in the Introduction) outputs a plausible ranking. In any case, $ \frac{1}{\sigma_2(\bI-\bP)}$ can be as small as $1-\epsilon$ (for example, the a walk on the complete graph), However, it can be quite large as well (for example, Gambler's Ruin). In simple terms, our result says that the randomly determined $\phi$ is sufficiently close to $\bar \phi$ provided (1) the spread of the degrees is not too large, (2) the spread of the components in $\ol \phi$ is not too large, and (3) the norm $ \frac{1}{\sigma_2(\bI-\bP)}$ not too large.

\section{Cauchy Properties of Random Walks}\label{Cauch}
The goal of this section is to show that if the one-step total change in a random walk is small, then, in fact, the distribution is ``close'' to the stationary distribution. How close is dependent on the ``singular gap:'' $\frac{1}{\sigma_2(\bI-\bP)}_2$. 

 In particular, we prove the following:

\begin{prop}\label{cauchy}
Let $\bv \in \mathbb{R}^n$ be any nonnegative vector with $\bv^* \mathbf{1} = 1$, and $\phi$ be the unique stationary distribution vector for the random walk with probability transposition matrix $\bP$. Define the vector $\varepsilon^* := \bv^* - \bv^*\bP$, then 

\[ \| \phi - \bv \| \le \|\mathbf{\varepsilon}\|  \frac{1}{\sigma_2(\bI-\bP)} \]

where $( \cdot) ^+$ denotes the psuedoinverse and $\| \cdot \|$ is any vector norm (or its induced matrix norm).
\end{prop}

In what follows,
for $\bx \in \mathbb{C}^{n}$,
we let $\bx = \sum_{i=1}^n \rho_i \bu_i$ 
be the left (generalized) eigenvector decomposition of $\bx$ for probability transition matrix $\bP$
with eigenvalues $1 = |\lambda_1| < |\lambda_2| \le \ldots  \le |\lambda_n|$
and (generalized) left eigenvectors $\bu_i^*$ corresponding to eigenvalues $\lambda_i$.

\begin{lem}\label{isperp}
Let $\bx, \by \in \mathbb{R}^{n}$. 
If $\bx^* \bP = \bx^* + \by^*$, 
then $\by \perp \mathbf{1}$.
\end{lem}

\begin{proof}
Since $\bP$ is a probability transition matrix the sum of $\bx^*$ is the same as $\bx^* \bP$, the sum of $\bx^* \bP-\bx^*$ must be 0. Hence, $\by^* = \bx^* \bP - \bx^*$ sums to 0.
\end{proof}

\begin{lem}\label{leftpright}
Let $\bA$ be a real $n \times n$ matrix.
Suppose $\bx$ is a left generalized eigenvector with eigenvalue $\lambda_i$ and $\by^*$ is a right (not generalized) eigenvector with eigenvalue $\lambda_j$. Then, if $\lambda_i \ne \lambda_j$, $\by^* \bx = 0.$
\end{lem}

\begin{proof}
Since $\bx$ is a generalized eigenvector, $\bx^* (\bA - \lambda_i \bI)^k = \mathbf{0}^*$ for some integer $k$. Consider $\bx^* (\bA - \lambda_i\bI)^k \by$. On one hand, we have $\bx^* (\bA - \lambda_i\bI)^k \by = \left(\bx^* (\bA - \lambda_i\bI)^k\right) \by = \mathbf{0}^* \by = 0$. On the other, $\bx^* (\bA - \lambda_i\bI)^k \by = \bx^* \left((\bA - \lambda_i\bI)^k \by\right) = \bx^*\by (\lambda_j-\lambda_i)^k$.
Hence, $0 =\bx^*\by (\lambda_j-\lambda_i)^k$. By hypothesis, $\lambda_i \ne \lambda_j$, so $\bx^*\by = 0$.
\end{proof}

\begin{lem}\label{perp0}
Let $\bx \in \mathbb{R}^{n}$ where $\bx = \sum_{i=1}^n \rho_i \bu_i$ 
is the left (generalized) eigenvector decomposition of $\bx$ for probability transition matrix $\bP$
Then, whenever $\bx^*\mathbf{1} = 0$,  $\rho_1 = 0$.
\end{lem}

\begin{proof}
\begin{eqnarray*}
0 = \bx ^* \mathbf{1} &=& \left(\sum_{i=1}^n \rho_i \bu_i\right)^* \mathbf{1} \\ 
&=& \left(\sum_{i=1}^n \rho_i \bu_i^* \mathbf{1} \right) \\
&=& \rho_1 \phi^* \mathbf{1} + \left(\sum_{i=2}^n \rho_i \bu_i^* \mathbf{1} \right) \\
&=& \rho_1 \phi^* \mathbf{1}  \\
&=&  \rho_1 \\
\end{eqnarray*}

where the fourth line follows from Lemma \ref{leftpright}.
\end{proof}

\begin{lem}\label{0pi}
For any vector $\bx \in \mathbb{R}^{n}$ with $\bx^* \bone = 0$. Suppose $\bP$ has a unique stationary distribution, then, 
, 
then the sum $$\bx^* \bI + \bx^* \bP + \bx^* \bP^2 + \ldots + \bx^* \bP^k + \ldots$$ is well-defined 
and is equal to  $\bx^* (\bI - \bP)^+$ 
where $(\cdot)^+$ is the pseudoinverse.
\end{lem}

\begin{proof}

Let $\bP = \bU \Lambda \bU^{-1}$ be the Jordan canonical form of $\bP$. Let $\hat \bP =  \bU \hat \Lambda \bU^{-1}$, where $\hat \Lambda$ replaces the 1 entry (from the eigenvalue of the stationary distribution) with 0 . By the spectral mapping theorem, 
since the eigenvalues of $\hat \bP$ are strictly within the unit circle in the complex plane, $(\bI-\hat \bP)^{-1} = \sum_{k=0}^{\infty} \hat \bP^k$. As a consequence of Lemma \ref{perp0}, since $\bone$ is a right eigenvector and $\bx^*\bone = 0$, we have that $\bx^* \bP = \bx^* \hat \bP$. Therefore $\bx^* \sum_{k=0}^{\infty} \hat \bP^k = \sum_{k=0}^{\infty} \bx^* \hat \bP^k =  \sum_{k=0}^{\infty} \bx^* \bP^k$.

It remains to show that $(\bI-\hat \bP)^{-1} = (\bI-\bP)^+$.  We apply the definition of psuedoinverse:
\begin{eqnarray*}
 (\bI-\bP) (\bI-\hat \bP)^{-1}  (\bI-\bP) &=& \bU (\bI - \Lambda)  (\bI - \hat \Lambda)^{-1}  (\bI - \Lambda) \bU^{-1}\\ &=& \bU  (\bI - \hat \Lambda) \bU^{-1}\\ 
 &=& (\bI-\bP) 
 \end{eqnarray*}
 Likewise,  
 \begin{eqnarray*}
 (\bI-\hat \bP)^{-1}  (\bI-\bP) (\bI-\hat \bP)^{-1} &=& \bU (\bI - \hat \Lambda)^{-1}  (\bI - \Lambda) (\bI - \hat \Lambda)^{-1} \bU^{-1}\\
  &=& \bU  (\bI - \hat \Lambda)^{-1} \bU^{-1}\\
   &=& (\bI-\hat \bP)^{-1}. 
   \end{eqnarray*}

\end{proof}

\begin{proof}[Proof of Proposition \ref{cauchy}]
$\phi^* = \lim_{k \to \infty} \bv^*\bP^k$. Since $\bv^*\bP = \bv^* + \varepsilon^*$, inductively we have $\bv^* \bP^k = \bv^* + \varepsilon^* + \varepsilon^* P + \cdots + \varepsilon^* \bP^{k-1}$. Since $\bv^*\bP = \bv^* + \varepsilon^*$, by Lemma \ref{isperp}, $\varepsilon^* \mathbf{1} = 0$. Hence by Lemma \ref{perp0}, the left eigenvector decomposition of $\varepsilon^*$ has $\rho_1 = 0$. By Lemma \ref{0pi}, the infinite sum $\bv^* + \varepsilon^* + \varepsilon^* \bP + \cdots + \varepsilon^* \bP^{k-1} + \cdots$ is well defined and is equal to $\bv^* + \varepsilon^* (I - \bP)^+$. It follows that $\phi^* - \bv^* =  \varepsilon^* (\bI - \bP)^+.$ Applying any vector norm, and noting that a unique stationary distribution implies $\sigma_2(\bI-\bP) \ne 0$, yields the desired inequality. 
\end{proof}

A simple method for computing the stationary distribution, as done in the PageRank algorithm, is to perform iterative matrix multiplication until the one-step total change is sufficiently small (i.e., the ``power method'') \cite{PageRankO}. Often, this method is justified by the spectral gap which implies an exponential convergence rate \cite{PageRankO}. While the spectral gap implies a quick \emph{long-term} convergence rate, the short-term behavior is determined by the singular values of the probability matrix and not the eigenvalues. In fact, there are examples of a Markov Chains that exhibit small one-step change yet it far from the stationary distribution. Further, there are examples of Markov chains where the random walk probabilities slowly take on several different distribution patterns before converging to the stationary distribution \cite{}. In this sense, it is na\"{i}ve to determine when to terminate the power method based upon the spectral gap and the one-step change. Rather, one should use the ``singular gap'', $ \frac{1}{\sigma_2(\bI-\bP)}$ instead of the spectral gap. In simple terms, the spectral gap determines the {\it long term convergence} whereas the singular gap determines the {\it short term convergence}. We give a more detailed discussion regarding PageRank in particular in Section \ref{examplePR}.

\section{Concentration of Expectation}\label{concen}
Our overall strategy will be to use Proposition \ref{cauchy}  to bound the expected stationary distribution against the the random stationary distribution. In doing so we must use various concentration inequalities.

\begin{lem}[Chernoff Bound] \cite{AandS}\label{Cher}
Let $X_i, \ldots, X_n$ be independent variables such that $X_i \in [0,1]$ for all $i = 1, \ldots, n$. Let $\overline{X} := \sum_{i=1}^n X_i$ and $\mu = \mathbb{E}\ol X$. Then,

\[ \mathbb{P} \left[  \left |\overline{X} - \mu \right| > \varepsilon \mu  \right ] < 2 \exp\left(\frac{-\varepsilon^2}{3} \mu \right)\]
\end{lem} 

\begin{lem}[Simultaneous Concentration of Degrees] \cite{AandS} \label{simcondeg}
Let $\ol d^{in}_i$ and $\ol d^{out}_i$ denote the average in- and out-degrees of vertex $i$ respectively. Then,

\[ \left|\ol d^{in}_i -  d^{in}_i \right| < \left( \ol d^{in}_i \right)^{2/3} \text{ and }  \left|\ol d^{out}_i -  d^{out}_i \right| <  \left( \ol d^{out}_i \right)^{2/3} \]

holds for all vertices $i$ simultaneously with probability at least 


 \[1 - 4 n  \exp\left(\frac{- d_{min}^{1/3}}{3} \right)\]
 
 where $d_{min} = \min [d_{min}^{in}, d_{min}^{out}].$

\end{lem}

\begin{proof}
Note that $d_i^{out} = \sum_{j=1}^{n} A_{ij}$ and $d_i^{in} = \sum_{j=1}^{n} A_{ji}$; both of which are sums of indicator random variables. Hence, by applying Lemma \ref{Cher}, with $\mu = d_i^{out}$ and $\varepsilon = \left( \frac{1}{ d_i^{out}} \right)^{2/3}$ individually for each vertex $i$ we have 

\[ \left|\ol d^{out}_i -  d^{out}_i \right| >  \left( \ol d^{out}_i \right)^{2/3} \] 

with probability at most

\[   2 \exp\left(\frac{- \left(d^{out}_i\right)^{1/3}}{3} \right) \]

Applying the union bound over all vertices, we have that $\left|\ol d^{out}_i -  d^{out}_i \right| <  \left( \ol d^{out}_i \right)^{2/3}$  for any vertex with probability at most $2 n \exp\left(\frac{- \left(d^{out}_i\right)^{1/3}}{3} \right)$

\end{proof}

\begin{lem}[Concentration of $ \ol \bD \bD^{-1} $]\label{DDbound}

\[ \| \bI - \ol \bD\bD^{-1} \|_2 \le {d_{min}^{-1/3}} \]

with probability at least 

 \[1 - 4 n  \exp\left(\frac{-d_{min}^{1/3}}{3} \right)\]

\end{lem}

\begin{proof}
By Lemma \ref{simcondeg},  $\left|\ol d^{out}_i -  d^{out}_i \right| <  \left( \ol d^{out}_i \right)^{2/3}$
for all $i$ simultaneously with probability at least  $1 - 4 n  \exp\left(\frac{- d_{min}^{1/3}}{3} \right)$. In which case:

\begin{eqnarray*}
 \left|\ol d^{out}_i -  d^{out}_i \right| &<&  \left( \ol d^{out}_i \right)^{2/3} \\
 \left| \frac{\ol d^{out}_i}{d^{out}_i} -  \frac{d^{out}_i}{d^{out}_i } \right| &<&  \left( \ol d^{out}_i \right)^{-1/3} \\ 
 \left| \frac{\ol d^{out}_i}{d^{out}_i} -  1 \right| &<&  \left( \ol d^{out}_i \right)^{-1/3} \\
  \left| \frac{\ol d^{out}_i}{d^{out}_i} -  1 \right| &<&  \left( \ol d^{out}_{min} \right)^{-1/3} \\
 \end{eqnarray*}
 
Since the entries of $\ol \bD \bD^{-1}$ are precisely $\frac{\ol d^{out}_i}{d^{out}_i}$, this completes the proof.

\end{proof}

\begin{lem}[Bound on $\|\bA\|$] \label{Abound}
With probability at least $1 - 4 n  \exp\left(\frac{- d_{min}^{1/3}}{3} \right)$,
 \[ \| \bA \| \le (1+ d_{min}^{-1/3}) \sqrt{\ol d^{out}_{max} \ol d^{in}_{max}} \]
\end{lem}

\begin{proof}
$\|\bA\|_2^2 \le \|\bA\|_1 \|\bA\|_\infty$. Note $\|\bA\|_\infty$ is $d^{in}_{max}$ and $\|\bA\|_1$ is $d^{out}_{max}$. By Lemma \ref{simcondeg}, $d^{out}_{max} = \ol d^{out}_{max} (1 \pm \varepsilon)$ where $\varepsilon = \ol d_{min}^{-1/3}$, and  $d^{in}_{max} = \ol d^{in}_{max} (1 \pm \varepsilon)$ with probability at least as given in the hypothesis. Therefore, $\|\bA\|_2^2 \le  d^{out}_{max}  d^{in}_{max} (1 \pm \varepsilon)^2$.
\end{proof}

\begin{lem}[Bound on $\|\bA-\ol\bA\|$] \label{A-Abound}
With probability at least $1 - 4 n  \exp\left(\frac{- d_{min}^{1/3}}{3} \right)$,
 \[ \| \bA - \ol \bA\| \le (1+ d_{min}^{-1/3}) \sqrt[4]{\ol d^{out}_{max} \ol d^{in}_{max}} \]
\end{lem}

\begin{proof}
The proof follows similarly to Lemma \ref{Abound}. $\|\bA-\ol \bA\|_2^2 \le \|\bA-\ol \bA\|_1 \|\bA-\ol \bA\|_\infty$. Note $\|\bA-\ol \bA\|_\infty$ is $|\ol d^in_{max} - d^{in}_{max}|$ and $\|\bA\|_1$ is $\ol d^{out}_{max} - d^{out}_{max}$. By Lemma \ref{simcondeg}, $|\ol d^{in}_{max} - d^{in}_{max}| =  \sqrt{\ol d^{out}_{max}} (1 \pm \varepsilon)$ where $\varepsilon = \ol d_{min}^{-1/3}$, and  $|\ol d^{out}_{max} - d^{out}_{max}| = \ol d^{out}_{max} (1 \pm \varepsilon)$ with probability at least as given in the hypothesis. Therefore, $\|\bA-\ol\bA\|_2^2 \le  \sqrt{d^{out}_{max}  d^{in}_{max} }(1 \pm \varepsilon)^2$.
\end{proof}

\section{Proof of Main Result}\label{mainproof}

\begin{lem}\label{phimaxhalf}
\[ \| \phi \| \le \phi_{max}^{1/2} \]
\end{lem}

\begin{proof}
Upon maximizing $\| \phi \|$ under the constraint $ 0 \le \phi_i \le \phi_{max}$ by applying Lagrange multipliers, it follows that $\| \phi \|$ is maximized when as many entries as possible are $\phi_{max}$. However, since $\sum_i \phi_i = 1$, there can be at most $1/\phi_{max}$ such entries. Hence, $\| \phi \|^2 \le \phi_{max}^2 \frac{1}{\phi_{max}} = \phi_{max}$.
 \end{proof}

\begin{proof}[Proof of Theorem \ref{mainresult}]
By applying Proposition \ref{cauchy}, with $\bv^* = \ol \phi$ we have,

\[ \| \phi - \ol \phi \| \le \|\ol \phi - \ol \phi\bP\| \frac{1}{\sigma_2(\bI-\bP)}\]

Hence, when considering the chi-squared norm with respect to $\ol \phi$:

\[ \| \phi - \ol \phi \|_{2,\ol \phi} \le \|\ol \phi - \ol \phi\bP\| ~\frac{1}{\sigma_2(\bI-\bP)}~\| \ol \Phi^{-1/2}\| \]

Since $\| \ol \Phi^{-1/2}\| = \ol \phi_{min}^{-1/2}$, it suffices to show that 

\[ \|\ol \phi - \ol \phi\bP\| \le  \ol \phi_{max}^{1/2}    \left(1+ \ol d_{min}^{-1/3}\right) \ol d_{min}^{-1} \left[ \sqrt[4]{\ol d^{out}_{max} \ol d^{in}_{max}}+  {\ol d_{min}^{-1/3}}\sqrt{\ol d^{out}_{max} \ol d^{in}_{max}} \right ]\]

Let us begin:

\begin{eqnarray*}
\| \ol \phi - \ol \phi\bP \| &=& \| \ol \phi - \ol \phi\bD^{-1}\bA \|  \\
&=& \| \ol \phi~ \ol \bD^{-1} \ol \bA - \ol \phi\bD^{-1}\bA \|  \\
&=& \| \ol \phi~ \ol \bD^{-1} \ol \bA -   \ol \phi~ \ol \bD^{-1} \bA  +  \ol \phi~ \ol \bD^{-1} \bA -  \ol \phi \bD^{-1} \bA \| \\
&\le& \| \ol \phi~ \ol \bD^{-1} \ol \bA -   \ol \phi ~\ol \bD^{-1} \bA \| + \| \ol \phi~ \ol \bD^{-1} \bA -  \ol \phi~ \bD^{-1} \bA\| \\
&=& \| \ol \phi~ \ol \bD^{-1} (\ol \bA -\bA) \| + \| \ol \phi~ (\ol \bD^{-1} - \bD^{-1} ) \bA\| \\
&\le& \| \ol \phi\| \left( \|\ol \bD^{-1}\| \|(\ol \bA -\bA) \| +  \| \ol \bD^{-1} (\ol \bI - \ol \bD \bD^{-1})\| \|\bA\| \right) \\
&\le& \| \ol \phi\| \left( \|\ol \bD^{-1}\| \|(\ol \bA -\bA) \| +  \| \ol \bD^{-1}\| \| \ol \bI - \ol \bD \bD^{-1}\| \|\bA\| \right) \\
&\le& \ol \phi_{max}^{1/2}  \ol d_{min}^{-1} \left[   \left(1+ \ol d_{min}^{-1/3}\right) \sqrt[4]{\ol d^{out}_{max} \ol d^{in}_{max}}~ \right] \\
&&  +~ \ol \phi_{max}^{1/2}  {\ol d_{min}^{-1}}  {\ol d_{min}^{-1/3}}  (1+ \ol d_{min}^{-1/3}) \sqrt{\ol d^{out}_{max} \ol d^{in}_{max}}
\end{eqnarray*}

where the second line follows as $\ol \phi$ is an eigenvector of $\ol \bP$ and the last line follows with probability at least $1 - 4 n  \exp\left(\frac{- d_{min}^{1/3}}{3} \right)$,
by noting $\| \ol \bD^{-1} \| = \ol d_{min}^{-1}$,
by applying Lemma \ref{phimaxhalf} to $\|\phi\|$, Lemma \ref{Abound} to $\|\bA\|$, Lemma \ref{DDbound} to $\| \ol \bI - \ol \bD \bD^{-1}\|$, and Lemma \ref{A-Abound} to $\| \bA - \ol \bA\|$.
\end{proof}

%
%

\section{Example: Gambler's Ruin} \label{examplesGR}

Gamblers ruin is classical example of a Markov chain. Let $G$ be a graph with vertices $1, \ldots, n$ with $i \to j$ only if $j-i = 1, 2$ or $-1$. In the context of gambling, a random walk on $G$ represents sequence of bets made by the gambler where the label of each vertex represents the amount of money the gambler has at any point in time.

Under this random walk, in order to get to vertex $n$ after reaching vertex $1$, the gambler must increase his money at least twice as often as it decreases.  However, the probability of decreasing is twice the probability of increasing, therefore, stationary distribution of vertex $n$ must be exponentially small.

To make this example random, suppose, additionally, an arc is placed from $1 \to n$, $n\to 1$ and $2 \to n$ each independently with probability $1/2$. If all there arcs are added, then in fact the graph is regular, and hence the stationary distribution is uniform. As a result, with probability of at least 1/8, the stationary distribution is uniform, and with probability of at least 1/8, the stationary distribution has exponentially small entries. Simple matrix calculations indicate that the stationary distribution of the ``expected`` random walk does not have exponentially small entries.

We now present a similar example illustrating the importance of the term $\| (\bI-\bP)^+ \|$ in Theorem \ref{mainresult}. We call it {\it Tourist's Ruin}. Let $G$ be a random directed graph on $n^2$ vertices labeled $(i,j)$ for $i, j = 1, \ldots, n$. 
We call the set of vertices which have the same first entry a {\it cluster}, and call two clusters with first entries which differ by $1 \mod n$, {\it adjacent clusters}. 
To generate the random graph, place an arc between vertices of the same cluster with probability 1/2.  Place an arc  between $(i,k)$ and $(j,k)$ with probability $\frac{k \log n}{n}$ if $i-j = 1 \mod n$ for some constant $k$. Hence, the graph consists of a cycle clusters.
Since the ``expected'' graph is regular, the stationary distribution of the``expected random walk'' is uniform. However, simple calculations show that the probability of the graph being not strongly-connected is bounded away from 0. This is an artifact of $\frac{1}{\sigma_2(\bI-\bP)}$ being large.

\section{Example: $G(n,p)$} \label{exampleGNP}

For undirected graphs, the classical Erd\"os-Reyni random graph model, $G(n,p)$, places an edge $\{i, j\}$ with probability $p$ for all pairs $i, j$ of $n$ vertices. This model is well-studied and has many implications in combinatorics and graph theory \cite{AandS}. However, in the context of Markov chains, $G(n,p)$ is not interesting as the degree distribution dictates the stationary distribution of the corresponding Markov chain- which is not the case for directed graphs.
Here, we investigate the Markoc chain on the directed variation of the Erd\"os-Reyni random graph model, $\overrightarrow{G}(n,p)$, where an arc $i \to j$ is placed with probability $p$ independently for all ordered pairs $(i,j)$. It is natural to believe that the since the degrees are concentrated, the stationary distribution must also be concentrated. However, it could be the case that even while all of the degrees are concentrated, the stationary may be skewed as in the previous example of Gambler's Ruin. Therefore, it is not obvious nor trivial that the stationary distribution of $\overrightarrow{G}(n,p)$ is nearly uniform. Using Theorem \ref{mainresult}, we can definitively show that the stationary distribution of $\overrightarrow{G}(n,p)$ is, in fact, approximately uniform.

\begin{cor}
If $p \gg n^{-1+k}$ for $k>0$, then the stationary distribution of $\overrightarrow{G}(n,p)$ converges to the uniform distribution under the chi-squared norm.
\end{cor}

\begin{proof}
Given that the ``expected'' probability transition matrix is uniform except for the identical diagonal entries the principal eigenvector of $\ol \bP$ is uniform. In addition, since  $p \gg n^k$ for any $k>0$, then Lemma \ref{simcondeg} applies, and all of the degrees are concentrated. Hence, using Theorem \ref{mainresult}, it suffices to show that $\sigma_2(\bI-\bP)$ is bounded away from 0.

Notice that $\bI - \bP = \bD^{-1} (\bD - \bA) = \bD^{-1} (np \bI + \mathbf{E} - \bA) = $ where $\bD=np\bI+\mathbf{E}$.  Hence $\sigma_2(\bI-\bP) \ge \sigma_1{\bD^{-1}} \sigma_2(np \bI + \mathbf{E} - \bA)$.

To bound $\sigma_2(np \bI + \mathbf{E} - \bA)$, We use the following result, adapted for our case, of Bai and Yin \cite{maxsig}:

\begin{thm}\cite{maxsig}
Let $\bM$ be a random asymmetric $n \times n$ matrix with i.i.d. entries with mean 0 and variance 1 and finite fourth moment. Then, $\|\bM\| \le 2\sqrt{n} (1+o(1))$ asymptotically almost surely.
\end{thm}

$np \bI + \mathbf{E} - \bA = np \bI + \mathbf{E} + p\bJ + (- p\bJ + \bA)$. Therefore, $\sigma_2(np \bI + \mathbf{E} - \bA) \ge \| np \bI \| - \|\mathbf{E}\| - \sigma_2(p\bJ) - \|\bA-p\bJ\|$. By applying the theorem above, we have asymptotically almost surely, $\frac{1}{\sqrt{p}} \|- p\bJ + \bA\| \le 2 \sqrt{n} (1+o(1))$. By the Chernoff Bound,  $\|\mathbf{E}\| \le (np)^{2/3}$. and since $\bJ$ is the all-ones matrix $\sigma_2(\bJ) = 0$. Altogether, we have asymptotically almost surely,

\begin{eqnarray*}
\sigma_2(\bD - \bA) &=& \sigma_2(np \bI + \mathbf{E} - \bA)\\
 &\ge& \| np \bI \| - \|\mathbf{E}\| - \sigma_2(p\bJ) - \|\bA-p\bJ\| \\
&\ge& np - (np)^{2/3}- 0 - 2 \sqrt{np}(1+o(1))\\
&=& np (1-o(1)) \\
\end{eqnarray*}

Therefore,
\begin{eqnarray*}
 \sigma_2(\bI-\bP)  &\ge& \sigma_1(\bD^{-1}) \sigma_2(\bD - \bA) \\
 &=& \frac{1}{ \ol d^{out}_{max}} (1-o(1)) \sigma_2(\bD - \bA)\\
&=& \frac{np (1-o(1))^2}{np}\\
&=& 1-o(1)\\
\end{eqnarray*}

Where the second and third lines follow from the concentration of degrees.
\end{proof}

\section{Example: PageRank} \label{examplePR}

PageRank is a special type of random walk on a graph. Given a constant $\alpha \in [0,1]$, at each step of the walk, a random walker currently at vertex $i$ has two options: with probability $1-\alpha$,  
 the walker follows an out-going arc chosen uniformly at random, and with probability $\alpha$, the walker jumps to another vertex chosen uniformly at random. However, if vertex $i$ has no out-going arcs, the walker employs the latter option. Hence, provided $d_{min}^{out} > 0$, the probability transition matrix for the PageRank Markov chain, $\mathbf{R}$, can be written as $\mathbf{R} = \frac{\alpha}{n-1} (\bJ - \bI) + (1-\alpha) \bP$ where $\bJ$ is the all-one matrix.
 
 The main advantage of PageRank is that the spectral gap (i.e., the largest nontrivial eigenvalue in modulus) is guaranteed to be at most $1-\alpha$. Hence, quick convergence in the long-term is guaranteed \cite{PageRankO}. However, the maximal one-step change is bounded by the singular values, or ``singular gap'', which may be quite small even with a large spectral gap.
 
While Theorem \ref{mainresult} only considers $(0,1)$-matrices, it can be quite easily adapted to other types of Markov chains, such as PageRank. (For example, by considering random $(\alpha/\ol d^{out}_i , 1+ \alpha/\ol d^{out}_i)$ matrices instead). PageRank gains an advantage over other Markov Chains with regard to Theorem \ref{mainresult} in that the new in- and out-degree of each vertex is at least $\alpha n$ {\it and} the stationary distribution of each vertex is at least $\alpha / n$. Hence, provided $\phi_{max} \ll n^-k$ for any $k>0$, and $\sigma_2(\bI - \mathbf{R})^{-1}$ is well behaved, the resulting PageRank on {\it almost any} matrix is concentrated. This is in contrast to Theorem \ref{mainresult} which can only apply to random graphs whose degree sequences are concentrated.

Bounding the singular gap for PageRank is not trivial. In fact, we present the following conjecture: 

\begin{conj}
For a directed graph where $d_{min}^{out} > 0$, the associated PageRank random walk with parameter $\alpha$ and probability transition matrix $\mathbf{R}$ obey
\[ \sigma_2(\bI - \mathbf{R})^{-1} \le \frac{k}{\alpha} \]
for some universal constant $k$.
\end{conj}

Our calculations indicate that $k \ge 1.65637$ using the example of Gambler's ruin in the previous sections. Specifically, while the spectral gap $\bI-\mathbf{R}$ is guaranteed to be at least $\alpha$, that is not necessarily true for the singular gap. Nonetheless, we conjecture that the parameter $\alpha$, like the spectral gap, does, in fact, give a bound on the singular gap for all PageRank transition matrices.

\section{Conclusion}\label{discuss}

\end{document}